\documentclass[12pt]{amsart}

\usepackage[matrix,arrow,curve,frame]{xy}
\usepackage{amsmath,amsthm,amssymb,enumerate}
\usepackage{latexsym}
\usepackage{amscd}
\usepackage[colorlinks=false]{hyperref}
\usepackage[mathscr]{eucal}

\newtheorem{thm}[subsection]{Theorem}

\newtheorem{cor}[subsection]{Corollary}
\newtheorem{lemma}[subsection]{Lemma}
\newtheorem{conj}[subsection]{Conjecture}

\theoremstyle{definition}

\numberwithin{equation}{section}

\def\phi{{\varphi}}

\def\ra{\rightarrow}

\def\cC{{\mathcal C}}
\def\cD{{\mathcal D}}
\def\cE{{\mathcal E}}
\def\cF{{\mathcal F}}

\def\cH{{\mathcal H}}
\def\cI{{\mathcal I}}
\def\cJ{{\mathcal J}}

\def\cV{{\mathcal V}}
\def\cW{{\mathcal W}}

\def\gg{{\mathfrak g}}
\def\gh{{\mathfrak h}}

\def\gl{{\mathfrak l}}

\def\go{{\mathfrak o}}
\def\gp{{\mathfrak p}}

\def\gs{{\mathfrak s}}

\newfont{\german}{eufm10}

\begin{document}
\pagestyle{plain}

\title
{Generalized parafermions of orthogonal type}

\author{Thomas Creutzig} 
\address{University of Alberta}
\email{creutzig@ualberta.ca}
\thanks{T. C. is supported by NSERC Discovery Grant \#RES0048511}

\author{Vladimir Kovalchuk} 
\address{University of Denver}
\email{vladimir.kovalchuk@du.edu}

\author{Andrew R. Linshaw} 
\address{University of Denver}
\email{andrew.linshaw@du.edu}
\thanks{A. L. is supported by Simons Foundation Grant 635650 and NSF Grant DMS-2001484}

\begin{abstract} There is an embedding of affine vertex algebras $V^k(\gg\gl_n) \hookrightarrow V^k(\gs\gl_{n+1})$, and the coset $\cC^k(n) = \text{Com}(V^k(\gg\gl_n), V^k(\gs\gl_{n+1}))$ is a natural generalization of the parafermion algebra of $\gs\gl_2$. It was called the algebra of generalized parafermions by the third author and was shown to arise as a one-parameter quotient of the universal two-parameter $\cW_{\infty}$-algebra of type $\cW(2,3,\dots)$. In this paper, we consider an analogous structure of orthogonal type, namely $\cD^k(n) = \text{Com}(V^k(\gs\go_{2n}), V^k(\gs\go_{2n+1}))^{\mathbb{Z}_2}$. We realize this algebra as a one-parameter quotient of the two-parameter even spin $\cW_{\infty}$-algebra of type $\cW(2,4,\dots)$, and we classify all coincidences between its simple quotient $\cD_k(n)$ and the algebras $\cW_{\ell}(\gs\go_{2m+1})$ and $\cW_{\ell}(\gs\go_{2m})^{\mathbb{Z}_2}$. As a corollary, we show that for the admissible levels $k = -(2n-2) + \frac{1}{2} (2 n + 2 m -1)$ for $\widehat{\gs\go}_{2n}$ the simple affine algebra $L_k(\gs\go_{2n})$ embeds in $L_k(\gs\go_{2n+1})$, and the coset is strongly rational.
As a consequence, the category of ordinary modules of $L_k(\gs\go_{2n+1})$ at such a level is a braided fusion category.
 \end{abstract}

\maketitle

\section{Introduction}
For $n\geq 1$, the natural embedding of Lie algebras $\gg\gl_n \hookrightarrow \gs\gl_{n+1}$ defined by 
$$ a \mapsto \bigg(\begin{matrix} a & 0\\ 0 & -\text{tr}(a) \end{matrix}\bigg),$$ induces a vertex algebra homomorphism \begin{equation} \label{embeddingA} V^k(\gg\gl_n) \hookrightarrow V^k(\gs\gl_{n+1}).\end{equation} The coset vertex algebra
\begin{equation} \cC^k(n) = \text{Com}(V^k(\gg\gl_n), V^k(\gs\gl_{n+1})) \end{equation} was called the algebra of {\it generalized parafermions} in \cite{LIII}. The reason for this terminology is that for $n=1$, $\cC^k(1)$ is isomorphic to the parafermion algebra $N^k(\gs\gl_2) = \text{Com}(\cH, \gs\gl_2)$, where $\cH$ denotes the Heisenberg algebra corresponding to the Cartan subalgebra $\gh \subseteq\gs\gl_2$.

By  Theorem 8.1 of \cite{LIII}, $\cC^k(n)$ is of type $\cW(2,3,\dots,  n^2+3n+1)$, i.e., it has a minimal strong generating set consisting of one field in each weight $2,3,\dots, n^2+3n+1$. This generalizes the case $n=1$, which appears in \cite{DLY}. When $k$ is a positive integer, \eqref{embeddingA} descends to a map of simple affine vertex algebras $L_k(\gg\gl_n) \hookrightarrow L_k(\gs\gl_{n+1})$, and the coset $\text{Com}(L_k(\gg\gl_n), L_k(\gs\gl_{n+1}))$ coincides with the simple quotient $\cC_k(n)$ of $\cC^k(n)$. By Theorem 13.1 of \cite{ACL},
we have an isomorphism \begin{equation} \label{typeAcoinc} \cC_k(n) \cong \cW_{\ell}(\gs\gl_k),\qquad \ell = -k+ \frac{k + n}{k + n +1}.\end{equation} In particular, $\cC_k(n)$ is {\it strongly rational}, that is, $C_2$-cofinite and rational. This generalizes the case $n=1$, which was proved earlier in \cite{ALY}.

A useful perspective on $\cC^k(n)$ is that these algebras all arise in a uniform way as quotients of the universal two-parameter $\cW_{\infty}$-algebra $\cW(c,\lambda)$ of type $\cW(2,3,\dots)$; see Theorem 8.2 of \cite{LIII}. This realization gives a nice conceptual explanation for the isomorphisms appearing in \eqref{typeAcoinc}. Each one-parameter quotient of $\cW(c,\lambda)$ corresponds to an ideal in $\mathbb{C}[c,\lambda]$, or equivalently, a curve in the parameter space $\mathbb{C}^2$ called the {\it truncation curve}. The truncation curves for $\cW^{\ell}(\gs\gl_m)$ and $\cC^k(n)$ are given by Equations 7.8 and 8.4 of \cite{LIII}, and the above isomorphisms correspond to intersection points on these curves.

The algebras $\cC^k(n)$ appear naturally as building blocks for affine vertex algebras of type $A$. It is convenient to replace $\cC^k(n)$ with $\tilde{\cC}^k(n) = \cH\otimes \cC^k(n)$, where $\cH$ is a rank one Heisenberg vertex algebra. Then we have $$\text{Com}(V^k(\gg\gl_{n-1}), V^k(\gg\gl_n)) \cong \tilde{\cC}^k(n-1),$$ so $V^k(\gg\gl_n)$ can be regarded as an extension of $V^k(\gg\gl_{n-1}) \otimes \tilde{\cC}^k(n-1)$. Iterating this procedure, we see that $V^k(\gg\gl_n)$ is an extension of 
\begin{equation} \label{gelfandts} \cH \otimes \tilde{\cC}^k(1)\otimes \tilde{\cC}^k(2) \otimes \cdots \otimes \tilde{\cC}^k(n-1).\end{equation} Note that if $k$ is a positive integer, the simple quotient $L_k(\gg\gl_n)$ is then an extension of 
 $$\cH \otimes \tilde{\cC}_k(1) \otimes \tilde{\cC}_k(2) \otimes \cdots \otimes \tilde{\cC}_k(n-1) \cong \cW_{\ell_1} (\gg\gl_k) \otimes \cW_{\ell_2} (\gg\gl_k)\otimes \cdots \otimes \cW_{\ell_n} (\gg\gl_k),$$ where $\ell_i = -k + \frac{k+n-i}{k+n+1-i}$. In \cite{ACL}, this was regarded as a noncommutative analogue of the Gelfand-Tsetin subalgebra of $U(\gg\gl_n)$. Similarly, we may regard the subalgebra  \eqref{gelfandts} as the universal version of this structure.

The algebras $\cC^k(n)$ also appear as building blocks for various $\cW$-(super)algebras. For example, an important conjecture of Ito \cite{I} asserts that the principal $\cW$-algebra $\cW^{\ell}(\gs\gl_{n+1|n})$ has a coset realization as
\begin{equation} \label{kazamasuzuki} \text{Com}(V^{k+1}(\gg\gl_n), V^k(\gs\gl_{n+1}) \otimes \cF(2n)),\end{equation} where $\cF(2n)$ denotes the rank $2n$ free fermion algebra, and $(\ell + 1)(k+n+1) = 1$. Ito's conjecture was stated in this form in \cite{CLII}, and these algebras have the same strong generating type by Lemma 7.12 of \cite{CLII}. In the case $n=1$, the conjecture clearly holds because both sides are isomorphic to the $N=2$ superconformal algebra. The first nontrivial case $n=2$ was proven in \cite{GL}. It was also shown in \cite{GL} that the coset \eqref{kazamasuzuki} is naturally an extension of $\cW^r(\gg\gl_n) \otimes \cC^k(n)$ for $r = -n + \frac{n+k}{n +k+1}$. An important ingredient in the proof of Ito's conjecture will be to show that $\cW^{\ell}(\gs\gl_{n+1|n})$ is indeed an extension of $\cW^r(\gg\gl_n) \otimes \cC^k(n)$. Note that $\cC^k(n)$ is itself a subalgebra of a $\cW$-superalgebra of $\gs\gl_{n+1|n}$ corresponding to a \emph{small hook-type} nilpotent element \cite{CLIII}.

\subsection{Generalized parafermion algebras of orthogonal type}
There are two different analogues of $\cC^k(n)$ in the orthogonal setting. We have natural embeddings $\gs\go_{2n}\hookrightarrow \gs\go_{2n+1} \hookrightarrow \gs\go_{2n+2}$, which induce homomorphisms of affine vertex algebras
\begin{equation} \label{embeddingBD} V^k(\gs\go_{2n}) \hookrightarrow V^k(\gs\go_{2n+1}) \hookrightarrow V^k(\gs\go_{2n+2}).\end{equation}
The cosets $\text{Com}(V^k(\gs\go_{2n}), V^k(\gs\go_{2n+1}))$ and $\text{Com}(V^k(\gs\go_{2n+1}), V^k(\gs\go_{2n+2}))$ both have actions of $\mathbb{Z}_2$, and we define 
 \begin{equation} \label{def:bdcosets} \cD^k(n) = \text{Com}(V^k(\gs\go_{2n}), V^k(\gs\go_{2n+1}))^{\mathbb{Z}_2},\quad  \cE^k(n) = \text{Com}(V^k(\gs\go_{2n+1}), V^k(\gs\go_{2n+2}))^{\mathbb{Z}_2}.\end{equation}
 Both these algebras arise as one-parameter quotients of the universal even spin $\cW_{\infty}$-algebra $\cW^{\rm ev}(c,\lambda)$ constructed recently by Kanade and the third author in \cite{KL}. Such quotients of $\cW^{\rm ev}(c,\lambda)$ are in bijection with a family of ideals $I$ in the polynomial ring $\mathbb{C}[c,\lambda]$, or equivalently, the truncation curves $V(I) \subseteq \mathbb{C}^2$. The main result in this paper is the explicit description of the truncation curve for $\cD^k(n)$ for all $n$; see Theorem \ref{main:realization}. The proof is based on the coset realization of principal $\cW$-algebras of type $D$ and a certain level-rank duality appearing in \cite{ACL}, which implies that
\begin{equation} \label{typeDcoinc} \cD_{2m}(n) \cong \cW_{\ell}(\gs\go_{2m})^{\mathbb{Z}_2},\qquad \ell = -(2m-2) + \frac{2m+2n-2}{2m+2n-1}.\end{equation} Here $\cD_{2m}(n)$ denotes the simple quotient of $\cD^{2m}(n)$. This is analogous to the isomorphisms \eqref{typeAcoinc} in type $A$. Since a similar coset realization of type $B$ principal $\cW$-algebras is not  available, we are currently unable to obtain an explicit description of $\cE^k(n)$, and in this paper we only study $\cD^k(n)$.

As in type $A$, there is a similar description of affine vertex algebras of orthogonal type  as extensions of Gelfand-Tsetlin type subalgebras. Clearly $V^k(\gs\go_{2n+2})$ is an extension of $$\cH \otimes \cD^k(1) \otimes \cE^k(1)  \otimes \cD^k(2) \otimes \cE^k(2) \otimes \cdots \otimes \cD^k(n-1)\otimes \cE^k(n-1)\otimes \cD^k(n)\otimes \cE^k(n),$$ and similarly, $V^k(\gs\go_{2n+1})$ is an extension of 
$$\cH \otimes \cD^k(1) \otimes \cE^k(1)  \otimes \cD^k(2) \otimes \cE^k(2) \otimes \cdots \otimes \cD^k(n-1)\otimes \cE^k(n-1) \otimes \cD^k(n).$$
Additionally, $\cD^k(n)$ is a building block for various $\cW$-(super)algebras. For example, 
consider the principal $\cW$-superalgebra $\cW^\ell(\go\gs\gp_{2n|2n})$ where $(\ell + 1)(k+2n-1) =1$. Note that $1$ and $2n-1$ are the dual Coxeter numbers of $\go\gs\gp_{2n|2n}$ and $\gs\go_{2n+1}$, respectively. The free fermion algebra $\cF(2n)$ carries an action of $L_1(\gs\go_{2n})$, and it is expected that
\begin{equation} \label{kazamasuzukiBD} \cW^\ell(\go\gs\gp_{2n|2n}) \cong \text{Com}(V^{k+1}(\gs\go_{2n}), V^k(\gs\go_{2n+1}) \otimes \cF(2n)).\end{equation} This algebra appears in physics in the duality of $N=1$ superconformal field theories and higher spin supergravities \cite{CHR, CV}, and this conjecture appeared in this context. Note that central charges coincide. It is apparent that the coset appearing in \eqref{kazamasuzukiBD} is an extension of $\cW^r(\gs\go_{2n}) \otimes \cD^k(n)$ where $r = -(2n-2) + \frac{k+2n-2}{k+2n-1}$. As in the case of Ito's conjecture, an important step in the proof of \eqref{kazamasuzukiBD} will be to show that $\cW^\ell(\go\gs\gp_{2n|2n})$ is also an extension of this structure.

\subsection{Applications} The first application of our main result is to classify all isomorphisms between the simple quotient $\cD_k(n)$ and the simple algebras $\cW_{\ell}(\gs\go_{2m+1})$ and $\cW_{\ell}(\gs\go_{2m})^{\mathbb{Z}_2}$. Using results of \cite{KL}, this can be achieved by finding the intersection points between the truncation curve for $\cD^k(n)$, and the truncation curves for $\cW^{\ell}(\gs\go_{2m+1})$ and $\cW^{\ell}(\gs\go_{2m})^{\mathbb{Z}_2}$, respectively. In the type $A$ case, we find only one family of points where $\cC_k(n)$ is isomorphic to a strongly rational $\cW$-algebra of type $A$; these appear in \eqref{typeAcoinc}. In the orthogonal setting, the situation is more interesting. In addition to the isomorphisms \eqref{typeDcoinc} when $k$ is a positive integer, we also find that for $k = -(2n-2) + \frac{1}{2} (2 n + 2 m -1)$, we have an embedding of simple affine vertex algebras $L_k(\gs\go_{2n}) \rightarrow L_k(\gs\go_{2n+1})$, and an isomorphism
$$\cD_k(n) = \text{Com}(L_k(\gs\go_{2n}) , L_k(\gs\go_{2n+1}))^{\mathbb{Z}_2} \cong \cW_{\ell}(\gs\go_{2m+1}),\qquad \displaystyle  \ell = -(2m-1) + \frac{2 m + 2 n -1}{2 m + 2 n+1}.$$ Since $\ell$ is a nondegenerate admissible level for $\gs\go_{2m+1}$, $\cW_{\ell}(\gs\go_{2m+1})$ is strongly rational \cite{ArI,ArII}. These are new examples of cosets of non-rational vertex algebras by admissible level affine vertex algebras, which are strongly rational.

This coset is also closely related to level-rank duality. Recall that $2n(2m+1)$ free fermions carry an action of $L_{2n}(\gs\go_{2m+1}) \otimes L_{2m+1}(\gs\go_{2n})$. The levels shifted by the respective dual Coxeter numbers are $2n+2m-1$ in both cases. 
Therefore $L_k(\gs\go_{2n+1})$ is an extension of $L_k(\gs\go_{2n}) \otimes  \cW_{\ell}(\gs\go_{2m+1})$, where 
$\ell = -(2m-1) + \frac{2 m + 2 n -1}{2 m + 2 n+1}$, i.e., both levels $k$ and $\ell$ shifted by the respective dual Coxeter numbers are of the form $(2m+2n-1)/v$ for $v =2$ and $v = 2 + 2m+2n-1$. In particular, the shifted levels have the same numerator as the original level-rank duality and the two denominators only differ by a multiple of the numerator. Note that under certain vertex tensor category assumptions the tensor product of two vertex algebras can be extended to a larger vertex algebra with a certain multiplicity freeness condition if and only if the two vertex algebras have subcategories that are braid-reversed equivalent, see \cite[Main Thm. 3]{CKM2} for the precise statement. Applied to our setting, this means that there are vertex algebra extensions of $L_k(\gs\go_{2n})$ and $\cW_{\ell}(\gs\go_{2m+1})$ that have subcategories of modules that are braid-reversed equivalent. 

The theory of vertex algebra extensions, especially \cite[Thm. 5.12]{CKM2}, then implies that the category of ordinary modules of $L_k(\gs\go_{2n+1})$ at 
level $k = -(2n-2) + \frac{1}{2}(2n+2m-1)$ is fusion, i.e. a rigid braided semisimple tensor category. This proves special cases of Conjecture 1.1 of \cite{CHY}.  

Finally, our rationality results for $\cD_k(n)$ suggest the existence of a new series of principal $\cW$-superalgebras of $\go\gs\gp_{2n|2n}$ which are strongly rational. By Corollary 14.2 of \cite{ACL}, the coset $\text{Com}(L_{k+1}(\gs\go_{2n}), L_k(\gs\go_{2n+1}) \otimes \cF(2n))$ is strongly rational when $k$ is a positive integer. In view of the conjectured isomorphism \eqref{kazamasuzukiBD}, this implies that for $k$ a positive integer and $\ell$ satisfying $(\ell + 1)(k+2n-1) =1$, $\cW_\ell(\go\gs\gp_{2n|2n})$ is strongly rational. Similarly, it follows from Corollary 1.1 of \cite{CKM2} that for $k = -(2n-2) + \frac{1}{2} (2 n + 2 m -1)$ and $\ell$ satisfying $(\ell + 1)(k+2n-1) =1$, the coset $\text{Com}(L_{k+1}(\gs\go_{2n}), L_k(\gs\go_{2n+1}) \otimes \cF(2n))$ is again strongly rational. This motivates the following 
\begin{conj} For $k = -(2n-2) + \frac{1}{2} (2 n + 2 m -1)$ and $\ell$ satisfying $(\ell + 1)(k+2n-1) =1$,  $\cW_\ell(\go\gs\gp_{2n|2n})$ is strongly rational.
\end{conj}
The conjecture is true for the $N=2$ super Virasoro algebra, i.e. the case $n=1$ \cite{Ad}. Otherwise strong rationality for principal $\cW$-superalgebras of orthosymplectic type is completely open. There is, however, a $C_2$-cofiniteness results in the case of $\go\gs\gp_{2|2n}$ \cite[Cor. 5.19]{CGN}.

\section{Vertex algebras}
We shall assume that the reader is familiar with vertex algebras, and we use the same notation and terminology as the papers \cite{LIII,KL}. We first recall the universal two-parameter vertex algebra $\cW^{\text{ev}}(c,\lambda)$ of type $\cW(2,4,\dots)$, which was recently constructed in \cite{KL}. It is defined over the polynomial ring $\mathbb{C}[c,\lambda]$ and is generated by a Virasoro field $L$ of central charge $c$, and a weight $4$ primary field $W^4$, and is strongly generated by fields $\{L, W^{2i}|\ i \geq 2\}$ where $W^{2i} = W^4_{(1)} W^{2i-2}$ for $i\geq 3$. The idea of the construction is as follows. 
\begin{enumerate}
\item All structure constants in the OPEs of $L(z) W^{2i}(w)$ and $W^{2j}(z) W^{2k}(w)$ for $2i \leq 12$ and $2j+2k \leq 14$, are uniquely determined as elements of $\mathbb{C}[c,\lambda]$ by imposing the Jacobi identities among these fields. 
\item This data uniquely and recursively determines all OPEs $L(z) W^{2i}(w)$ and $W^{2j}(z) W^{2k}(w)$ over the ring $\mathbb{C}[c,\lambda]$ if a certain subset of Jacobi identities are imposed. 
\item By showing that the algebras $\cW^k(\gs\gp_{2m})$ all arise as one-parameter quotients of $\cW^{\text{ev}}(c,\lambda)$ after a suitable localization, we show that all Jacobi identities hold. Equivalently, $\cW^{\text{ev}}(c,\lambda)$ is freely generated by the fields $\{L, W^{2i}|\ i \geq 2\}$, and is the universal enveloping algebra of the corresponding nonlinear Lie conformal algebra \cite{DSK}.
\end{enumerate}

$\cW^{\mathrm{ev}}(c,\lambda)$ is simple as a vertex algebra over $\mathbb{C}[c,\lambda]$, but there is a certain discrete family of prime ideals $I = (p(c,\lambda)) \subseteq \mathbb{C}[c,\lambda]$ for which the quotient 
$$\cW^{{\rm ev}, I}(c,\lambda) = \cW^{\mathrm{ev}}(c,\lambda)/ I \cdot \cW^{\mathrm{ev}}(c,\lambda),$$ is not simple as a vertex algebra over the ring $\mathbb{C}[c,\lambda] / I$. We denote by $\cW^{\mathrm{ev}}_I(c,\lambda)$ the simple quotient of $\cW^{{\rm ev}, I}(c,\lambda)$ by its maximal proper graded ideal $\cI$. After a suitable localization, all one-parameter vertex algebras of type $\cW(2,4,\dots, 2N)$ for some $N$ satisfying some mild hypotheses, can be obtained as quotients of $\cW^{\mathrm{ev}}(c,\lambda)$ in this way. This includes the principal $\cW$-algebras $\cW^k(\gs\go_{2m+1})$ and the orbifolds $\cW^k(\gs\go_{2m})^{\mathbb{Z}_2}$. The generators $p(c,\lambda)$ for such ideals arise as irreducible factors of Shapovalov determinants, and are in bijection with such one-parameter vertex algebras. 

We also consider $\cW^{\mathrm{ev},I}(c,\lambda)$ for maximal ideals $$I = (c- c_0, \lambda- \lambda_0),\qquad c_0, \lambda_0\in \mathbb{C}.$$
Then $\cW^{\mathrm{ev},I}(c,\lambda)$ and its quotients are vertex algebras over $\mathbb{C}$. Given maximal ideals $I_0 = (c- c_0, \lambda- \lambda_0)$ and $I_1 = (c - c_1, \lambda - \lambda_1)$, let $\cW_0$ and $\cW_1$ be the simple quotients of $\cW^{\mathrm{ev},I_0}(c,\lambda)$ and $\cW^{\mathrm{ev},I_1}(c,\lambda)$. Theorem 8.1 of \cite{KL} gives a simple criterion for $\cW_0$ and $\cW_1$ to be isomorphic. Aside from a few degenerate cases, we must have $c_0 = c_1$ and $\lambda_0 = \lambda_1$. This implies that aside from the degenerate cases, all other coincidences among the simple quotients of one-parameter vertex algebras $\cW^{\mathrm{ev},I}(c,\lambda)$ and $\cW^{\mathrm{ev},J}(c,\lambda)$, correspond to intersection points of their truncation curves $V(I)$ and $V(J)$.

We shall need the following result which is analogous to Theorem 6.2 of \cite{LIII}.

\begin{thm} \label{refinedsimplequotient} Let $\cW$ be a vertex algebra of type $\cW(2,4,\dots, 2N)$ which is defined over some localization $R$ of $\mathbb{C}[c,\lambda] / I$, for some prime ideal $I$. Suppose that $\cW$ is generated by the Virasoro field $L$ and a weight $4$ primary field $W^4$. If in addition, the graded character of $\cW$ agrees with that of $\cW^{\rm ev}(c,\lambda)$ up to weight $13$, then $\cW$ is a quotient of $\cW^I(c,\lambda)$ after localization.
\end{thm}

\begin{proof} First, note that Theorem 3.10 of \cite{KL} holds without the simplicity assumption; see Remark 5.1 of \cite{LIII} for a similar statement in the case of the algebra $\cW(c,\lambda)$ of type $\cW(2,3,\dots)$. By Theorem 3.10 of \cite{KL}, it suffices to prove that the OPEs $L(z) W^{2i}(w)$ and $W^{2j}(z) W^{2k}(w)$ for $2i \leq 12$ and $2j+2k \leq 14$ in $\cW$ are the same as the corresponding OPEs in $\cW^{\rm ev}(c,\lambda)$ if the structure constants are replaced with their images in $R$. In this notation, $W^{2i}= W^4_{(1)} W^{2i-2}$ for $i\geq 3$. But this is automatic because the graded character assumption implies that there are no null vectors of weight $w\leq 13$ in the (possibly degenerate) nonlinear conformal algebra corresponding to $\{L, W^{2i}|\ 2 \leq i \leq N\}$. \end{proof}

\section{Generalized parafermions of orthogonal type} \label{section:genpara}

For $n\geq 1$, the natural embedding $\gs\go_{2n} \hookrightarrow \gs\go_{2n+1}$ induces a vertex algebra homomorphism $$V^k(\gs\go_{2n})  \rightarrow V^k(\gs\go_{2n+1}).$$ The action of $\gs\go_{2n}$ on $V^k(\gs\go_{2n+1})$ given by the zero modes of the generating fields integrates to an action of the orthogonal group $\text{O}_{2n}$. Therefore the coset 
$$\text{Com}(V^k(\gs\go_{2n}), V^k(\gs\go_{2n+1})) =  V^k(\gs\go_{2n+1})^{\gs\go_{2n}[t]}$$ has a nontrivial action of $\mathbb{Z}_2$. We define
\begin{equation} \cD^k(n) =\text{Com}(V^k(\gs\go_{2n}), V^k(\gs\go_{2n+1}))^{\mathbb{Z}_2}.\end{equation} 
It has Virasoro element $L^{\gs\go_{2n+1}} - L^{\gs\go_{2n}}$ with central charge 
\begin{equation} \label{ccgenpar} c  = \frac{k n (2 k + 2 n -3)}{(k + 2 n -2) (k + 2 n -1)} .\end{equation}
Note that in the case $n=1$, $\cD^k(n) \cong N^k(\gs\gl_2)^{\mathbb{Z}_2}$ which is of type $\cW(2,4,6,8,10)$ by Theorem 10.1 of \cite{KL}.

\begin{lemma} \label{stronggentype} For all $n\geq 1$, $\cD^k(n)$ is of type $\cW(2,4,\dots,2N)$ for some $N$ satisfying $N \geq 2n^2+3n$. We conjecture, but do not prove, that $N = 2n^2+3n$. Moreover, for generic values of $k$, $\cD^k(n)$ is generated by the weight $4$ primary field $W^4$. \end{lemma}

\begin{proof} By Theorem 6.10 of \cite{CLII}, we have $$\lim_{k\rightarrow \infty} \cD^k(n) \cong \cH(2n)^{\text{O}_{2n}},$$ and a strong generating set for $\cH(2n)^{\text{O}_{2n}}$ corresponds to a strong generating set for $\cD^k(n)$ for generic values of $k$. Here $\cH(2n)$ denotes the rank $2n$ Heisenberg vertex algebra. It was shown in \cite{LII}, Theorem 6.5, that $\cH(2n)^{\text{O}_{2n}}$ has the above strong generating type.  By Lemma 4.2 of \cite{LI}, the weights $2$ and $4$ fields generate $\cH(2n)^{\text{O}_{2n}}$. In fact, it is easy to check that only the weight $4$ field is needed, and that it can be replaced with a primary field which also generates the algebra. Finally, the statement that $\cD^k(n)$ inherits these properties of $\cH(2n)^{\text{O}_{2n}}$ for generic values of $k$ is also clear; the argument is similar to the proof of Corollary 8.6 of \cite{CLI}. \end{proof}

\begin{cor} \label{cor:gp} For all $n\geq 1$, there exists an ideal $K_n\subseteq \mathbb{C}[c,\lambda]$ and a localization $R_n$ of $\mathbb{C}[c,\lambda] / K_n$ such that $\cD^k(n)$ is the simple quotient of $\cW^{{\rm ev},K_n}_{R_n}(c,\lambda)$.
\end{cor}

\begin{proof} This holds for $n=1$ by Theorem 10.1 of \cite{KL}. For $n>1$, the simplicity of $\cD^k(n)$ as a vertex algebra over a localization of $\mathbb{C}[k]$ follows from the simplicity of $\cH(2n)^{\text{O}_{2n}}$, which follows from \cite{DLM}. In view of Theorems \ref{refinedsimplequotient} and \ref{stronggentype}, it then suffices to show that the graded characters of $\cD^k(n)$ and $\cW^{\text{ev}}(c,\lambda)$ agree up to weight $13$. This follows from Weyl's second fundamental theorem of invariant theory for $\text{O}_{2n}$ \cite{W}, since there are no relations among the generators of weight less than $4n^2+6n+2$. \end{proof}

\begin{thm} \label{main:realization} For all $n\geq 2$, $\cD^k(n)$ is isomorphic to a localization of the quotient $\cW^{\text{ev}}_{K_n}(c,\lambda)$, where the ideal $K_n\subseteq \mathbb{C}[c,\lambda]$ is described explicitly via the parametrization $k \mapsto (c_n(k), \lambda_n(k))$ given by
\begin{equation} \label{eq:realization} \begin{split}  c_n(k)  & = \frac{k n (2 k + 2 n -3)}{(k + 2 n -2) (k + 2 n -1)},  \ \ \  \lambda_n(k)  = \frac{(k + 2 n -2) (k + 2 n -1) p_n(k)}{7 (k-2) ( k + n -1) (2 n-1) q_n(k) r_n(k)},
\\  p_n(k)  & = -112 + 188 k - 62 k^2 - 26 k^3 + 12 k^4 + 744 n - 1336 k n + 857 k^2 n - 252 k^3 n 
\\ &\quad  + 36 k^4 n - 1720 n^2 + 2534 k n^2 - 1198 k^2 n^2 + 188 k^3 n^2 + 1632 n^3 - 1544 k n^3 \\ & \quad+ 304 k^2 n^3 
  - 544 n^4 + 152 k n^4,
\\ q_n (k)  & =  20 - 19 k + 6 k^2 - 42 n + 28 k n + 28 n^2,
\\  r_n(k) & = 44 - 66 k + 22 k^2 - 132 n + 73 k n + 10 k^2 n + 88 n^2 + 10 k n^2.
\end{split}
\end{equation} 
\end{thm}

\begin{proof}
Let $n$ be fixed. In view of Corollary \ref{cor:gp} and the fact that all structure constants in $\cD^k(n)$ are rational functions of $k$, there is some rational function $\lambda_n(k)$ of $k$ such that $\cD^k(n)$ is obtained from $\cW^{\rm ev}(c,\lambda)$ by setting $\displaystyle c =c_n(k) $ and $\lambda = \lambda_n(k)$, and then taking the simple quotient. It is not obvious yet that $\lambda_n(k)$ is a rational function of $n$ as well.

For $k$ a positive integer, it is well known \cite{KW} that the map $V^k(\gs\go_{2n}) \ra V^k(\gs\go_{2n+1})$ descends to a homomorphism of simple algebras $L_k(\gs\go_{2n}) \ra L_k(\gs\go_{2n+1})$. Letting $\cD_k(n)$ denote the simple quotient of $\cD^k(n)$, it is apparent from Lemma 2.1 of \cite{ACKL} and Theorem 8.1 of \cite{CLII} that $\text{Com}(L_k(\gs\go_{2n}),L_k(\gs\go_{2n+1}))$ is simple and coincides with the simple quotient of $\text{Com}(V^k(\gs\go_{2n}),V^k(\gs\go_{2n+1}))$. Moreover, taking $\mathbb{Z}_2$-invariants preserves simplicity, hence
$$\cD_k(n) \cong \text{Com}(L_k(\gs\go_{2n}),L_k(\gs\go_{2n+1}))^{\mathbb{Z}_2}.$$

Next, by Corollary 1.3 of \cite{ACL}, for all $n\geq 1$ and $m\geq 2$, we have an isomorphism 
\begin{equation} \label{eq:ACL1} \big(\big(L_{2m}(\gs\go_{2n+1}) \oplus \mathbb{L}_{2m}(2m \omega_1)\big)^{\gs\go_{2n}[t]}\big)^{\mathbb{Z}_2 \times \mathbb{Z}_2} \cong \cW_{\ell}(\gs\go_{2m}),\quad \ell = -(2m-2) + \frac{2n+2m-2}{2n+2m-1}.\end{equation} In this notation, $\omega_1$ denotes the first fundamental weight of $\gs\go_{2n+1}$ and $\mathbb{L}_{2m}(2m \omega_1)$ denotes the simple quotient of the corresponding Weyl module.

Note that $\big(L_{2m}(\gs\go_{2n+1})^{\gs\go_{2n}[t]}\big)^{\mathbb{Z}_2} =  \cD_{2m}(n)$ is manifestly a subalgebra of the left hand side of \eqref{eq:ACL1}. Also, the lowest-weight component of $\mathbb{L}_{2m}(2m \omega_1)$ has conformal weight $m$. If $m>4$, the left-hand side then has a unique primary weight $4$ field which lies in $\cD_{2m}(n)$. Similarly, since $\cW_{\ell}(\gs\go_{2m})$ has strong generators in weights $2,4,\dots, 2m$ and $m$, for $m>4$ the right hand side has a unique primary weight $4$ field, which lies in the $\mathbb{Z}_2$-orbifold $\cW_{\ell}(\gs\go_{2m})^{\mathbb{Z}_2}$.

Since $\cD^k(n)$ is generated by the weight $4$ field as a one-parameter vertex algebra, the weight $4$ field must generate $\cD_{2m}(n)$ for all $m$ sufficiently large. By Corollary 6.1 of \cite{KL}, $\cW^{\ell}(\gs\go_{2m})^{\mathbb{Z}_2}$ is generated by the weight $4$ field as a one-parameter vertex algebra; equivalently, this holds for generic values of $\ell$. By the same argument as Proposition A.4 of \cite{ALY}, the vertex Poisson structure on the associated graded algebra $\text{gr}\ \cW^{\ell}(\gs\go_{2m})$ with respect to Li's canonical filtration, is independent of $\ell$ for all noncritical values of $\ell$. In particular this holds for the subalgebra  $(\text{gr}\ \cW^{\ell}(\gs\go_{2m}))^{\mathbb{Z}_2} =  \text{gr}(\cW^{\ell}(\gs\go_{2m})^{\mathbb{Z}_2})$. It follows from the same argument as Proposition A.3 of \cite{ALY} that $\cW^{\ell}(\gs\go_{2m})^{\mathbb{Z}_2}$ is generated by the weights $2$ and $4$ fields for all noncritical values of $\ell$, and the same therefore holds for the simple quotient $\cW_{\ell}(\gs\go_{2m})^{\mathbb{Z}_2}$. Finally, for $\ell = -(2m-2) + \frac{2n+2m-2}{2n+2m-1}$, it is straightforward to verify that the Virasoro field can be generated from the weight $4$ field, so the weight $4$ field generates the whole algebra.

Therefore if $m$ is sufficiently large, we obtain  
\begin{equation} \label{eq:ACL2} \cD_{2m}(n) \cong \cW_{\ell}(\gs\go_{2m})^{\mathbb{Z}_2},\qquad \ell = -(2m-2) + \frac{2m+2n-2}{2m+2n-1}.\end{equation}
In fact, we will see later (Theorem \ref{coin:d}) that this holds for all $m\geq 2$.

Finally, the truncation curve that realizes $\cW_{\ell}(\gs\go_{2m})^{\mathbb{Z}_2}$ as a quotient of $\cW^{\rm ev}(c,\lambda)$ is given by Theorem 6.3 of \cite{KL}, and in parametric form by Equation (B.1) of \cite{KL}. In view of \eqref{eq:ACL2}, we must have $\lambda_n(2m) = \lambda_m(\ell)$ for $\ell = -(2m-2) + \frac{2n+2m-2}{2n+2m-1}$ for $m$ sufficiently large, where $\lambda_m(\ell)$ is given by Equation (B.1) of \cite{KL}. If follows that for infinitely many values of $k$, $\lambda_n(k)$ is given by the above formula \eqref{eq:realization}. Since $ \lambda_n(k)$ is a rational function of $k$, this equality holds for all $k$ where it is defined. This completes the proof. \end{proof}

\section{Coincidences}
In this section, we shall use Theorem \ref{main:realization} to classify all coincidences between the simple quotient $\cD_k(n)$ and the $\mathbb{Z}_2$-orbifold $\cW_{\ell}(\gs\go_{2m})^{\mathbb{Z}_2}$, as well as $\cW_{\ell}(\gs\go_{2m+1})$. We also classify all coincidences between $\cD_k(n)$ and $\cD_{\ell}(m)$ for $m \neq n$. 

\begin{thm} \label{coin:d} For $n\geq 1$ and $m\geq 2$, aside from the critical levels $k = -2n+2$ and $k=-2n+1$, and the degenerate cases $c = \frac{1}{2}, -24$, all isomorphisms $\cD_k(n) \cong \cW_{\ell}(\gs\go_{2m})^{\mathbb{Z}_2}$ appear on the following list:

\begin{enumerate}

\item $\displaystyle k = 2m,\qquad \ell = -(2m-2) + \frac{2 n + 2 m -2}{2 n + 2 m -1}$,

\item $\displaystyle k = -(2n-2) -\frac{2 n -1}{2 (m-1)},\qquad \ell = -(2m-2) + \frac{2 m - 2 n -1 }{2 ( m-1)}$,

\item $\displaystyle k = -(2n-2) + \frac{n - m}{m},\qquad \ell = -(2m-2) + \frac{m-n}{m}$.

\end{enumerate}

\end{thm}

\begin{proof} Recall first that $\cW_{\ell}(\gs\go_{2m})^{\mathbb{Z}_2}$ is realized as the simple quotient of $\cW^{\mathrm{ev}, J_{m}}(c,\lambda)$, where the ideal $J_m \subseteq \mathbb{C}[c,\lambda]$ is given in parametrized form by Equation (B.1) of \cite{KL}. First, we exclude the values of $k$ and $\ell$ which are poles of the functions $\lambda_n(k)$ given by \eqref{eq:realization}, and $\lambda_m(\ell)$ given by Equation (B.1) of \cite{KL}, since at these values, $\cD^k(n)$ and $\cW_{\ell}(\gs\go_{2m})^{\mathbb{Z}_2}$ are not quotients of $\cW^{\mathrm{ev}}(c,\lambda)$. For all other noncritical values of $k$ and $\ell$, $\cD^k(n)$ and $\cW_{\ell}(\gs\go_{2m})^{\mathbb{Z}_2}$ are obtained as quotients of $\cW^{\mathrm{ev},I_{n}}(c,\lambda)$ and $\cW^{\mathrm{ev}, J_{m}}(c,\lambda)$, respectively. By Corollary 8.2 of \cite{KL}, aside from the degenerate cases given by Theorem 8.1 of \cite{KL}, all other coincidences  $\cD_k(n) \cong \cW_{\ell}(\gs\go_{2m})^{\mathbb{Z}_2}$ correspond to intersection points on the truncation curves $V(K_n)$ and $V(J_m)$. A calculation shows that $V(K_n)\cap V(J_m)$ consists of exactly five points $(c,\lambda)$, namely, 
\begin{equation} \begin{split} &\bigg(-24, -\frac{1}{245}\bigg),\qquad \bigg(\frac{1}{2}, -\frac{2}{49}\bigg), \qquad \bigg(  \frac{m n (4 m + 2 n -3)}{(m + n -1) (2 m + 2 n -1)},\   \lambda_1  \bigg),
\\ & \bigg( -\frac{2 m n (3 - 4 m - 2 n + 4 m n)}{2 m - 2 n -1}, \  \lambda_2 \bigg), \quad \bigg( -\frac{(2 m n + m - 2 n) (2 m n -m - n)}{m - n}, \  \lambda_3 \bigg).\end{split} \end{equation}
Here
\begin{equation} 
\begin{split}
\lambda_1 & =  \frac{(m + n -1) (2 m + 2 n -1)g}{7 (m -1) (2 m + n -1) (2 n -1)gh},
\\  f & = -28 + 94 m - 62 m^2 - 52 m^3 + 48 m^4 + 186 n - 668 m n + 857 m^2 n - 504 m^3 n 
\\ & \quad+ 144 m^4 n  - 430 n^2 + 1267 m n^2 - 1198 m^2 n^2 + 376 m^3 n^2 + 408 n^3 - 772 m n^3
\\ & \quad + 304 m^2 n^3  - 136 n^4 + 76 m n^4,
\\ g &  = 10 - 19 m + 12 m^2 - 21 n + 28 m n + 14 n^2,
\\ h & = 22 - 66 m + 44 m^2 - 66 n + 73 m n + 20 m^2 n + 44 n^2 + 10 m n^2.
\end{split} \end{equation}

\begin{equation} 
\begin{split}
\lambda_2 & =   \frac{(1 - 2 m + 2 n) f}{7 (1 - 2 m +  2 m n) (-1 - 2 n + 4 m n) gh },
\\  f & = 14 - 33 m - 2 m^2 + 24 m^3 + 74 n - 404 m n +  873 m^2 n - 696 m^3 n + 144 m^4 n\\ 
&\quad  + 80 n^2   - 178 m n^2 - 260 m^2 n^2 + 452 m^3 n^2 - 112 m^4 n^2 - 24 n^3 + 264 m n^3 \\ 
&\quad - 348 m^2 n^3  + 256 m^3 n^3  - 64 m^4 n^3 + 72 m n^4 - 128 m^2 n^4 - 48 m^3 n^4 + 32 m^4 n^4,
 \\ g & = -10 + 19 m - 12 m^2 - 2 n + 22 m n - 8 m^2 n - 12 n^2 - 8 m n^2 + 8 m^2 n^2,
 \\ h & = 11 - 22 m + 22 n + 15 m n - 20 m^2 n - 10 m n^2 + 20 m^2 n^2.
 \end{split} \end{equation}

 \begin{equation} 
\begin{split}
\lambda_3 & =   \frac{(n-m) f}{7 (m-1) (2 n -1) (m - n + 2 m n) gh}, 
\\  f & = -34 m^3 + 19 m^4 + 68 m^2 n - 38 m^3 n - 22 m n^2 - 185 m^2 n^2 + 302 m^3 n^2 - 80 m^4 n^2  \\ 
& \quad- 12 n^3  + 204 m n^3 - 302 m^2 n^3 + 80 m^3 n^3 - 36 n^4 + 100 m n^4 - 40 m^2 n^4 - 40 m^3 n^4\\ &\quad + 16 m^4 n^4,
 \\ g & = -7 m^2 + 7 m n - 6 n^2 - 4 m n^2 + 4 m^2 n^2,
 \\ h & = -22 m -  5 m^2 + 22 n + 5 m n + 10 n^2 - 30 m n^2 + 20 m^2 n^2.  \end{split} \end{equation}

By Theorem 8.1 of \cite{KL}, the first two intersection points occur at degenerate values of $c$. By replacing the parameter $c$ with the levels $k$ and $\ell$, we see that the remaining intersection points yield the nontrivial isomorphisms in Theorem \ref{coin:d}. Moreover, by Corollary 8.2 of \cite{KL}, these are the only such isomorphisms except possibly at the values of $k, \ell$ excluded above. 

Finally, suppose that $k$ is a pole of the function $\lambda_n(k)$ given by \eqref{eq:realization}. It is not difficult to check that the corresponding values of $\ell$ for which $c_n(k) = c_m(\ell)$, are not poles of $\lambda_m(\ell)$. As above, $c_n(k)$ and $\lambda_n(k)$ are given by \eqref{eq:realization}, and $c_m(\ell)$ and $\lambda_m(\ell)$ are given by Equation (B.1) of \cite{KL}. It follows that there are no additional coincidences at the excluded points.
\end{proof}

Next, we classify the coincidences between $\cD_k(n)$ and $\cW_{\ell}(\gs\go_{2m+1})$.

\begin{thm} \label{coin:c} For $n\geq 1$ and $m\geq 2$, aside from the critical levels $k = -2n+2$ and $k=-2n+1$, and the degenerate cases $c = \frac{1}{2}, -24$, all isomorphisms $\cD_k(n) \cong \cW_{\ell}(\gs\go_{2m+1})$ appear on the following list:

\begin{enumerate}

\item $\displaystyle k = -(2n-2) + \frac{1}{2} (2 n + 2 m -1),\qquad \ell = -(2m-1) + \frac{2 m + 2 n -1}{2 m + 2 n+1}$,

\item $\displaystyle k = -(2n-2) + \frac{2 n - 2 m -1}{2 m+2},\qquad \ell = -(2m-1) +  \frac{2 m - 2 n +1}{2 m+2}$,

\item $\displaystyle k = -(2n-2)  -\frac{n}{m} ,\qquad \ell = -(2m-1) + \frac{m - n}{m}$,

\item $\displaystyle k = -(2n-2) -\frac{2 (n -1)}{2 m -1} ,\qquad \ell = -(2m-1) + \frac{2 m -1}{2 m - 2 n +1}$.

\item $\displaystyle k = -(2n-2) + \frac{2 (n - m -1)}{2 m+1},\qquad  \ell = -(2m-1) + \frac{2 m +1}{2 (m - n +1)}$.

\end{enumerate}

\end{thm}

\begin{proof} The argument is the same as the proof of Theorem \ref{coin:d}. First, $\cW_{\ell}(\gs\go_{2m+1})$ is realized as the simple quotient of $\cW^{{\rm ev}, I_m}(c,\lambda)$ where the ideal $I_m\subseteq \mathbb{C}[c,\lambda]$ is parametrized explicitly by Equation (A.3) of \cite{KL}. The above isomorphisms all arise from the intersection points between the truncation curves $V(K_n)$ for $\cD^k(n)$ and $V(I_m)$ for $\cW_{\ell}(\gs\go_{2m+1})$. A calculation shows that there are exactly $7$ intersection points: the degenerate points
$(\frac{1}{2},-\frac{2}{49})$ and $(-24, -\frac{1}{245})$, and the five nontrivial ones appearing above. One then has to rule out additional coincidences at the points where $\cD_k(n)$ does not arise as a quotient of $\cW^{\mathrm{ev}}(c,\lambda)$, namely, the poles of $\lambda_n(k)$. The details are straightforward and are left to the reader. \end{proof}

Finally, we classify all isomorphisms $\cD_k(m) \cong \cD_{\ell}(n)$ for $n\neq m$.

\begin{thm} \label{coin:self} For $m,n\geq 1$ and $n\neq m$, aside from the degenerate cases $c = \frac{1}{2}, -24$ and poles of $c_n(k)$, $\lambda_n(k)$ and $c_m(k)$, $\lambda_m(k)$ the complete list of isomorphisms $\cD_k(m) \cong \cD_{\ell}(n)$ is the following:

\begin{enumerate}

\item $\displaystyle k = -(2m-2) + \frac{2 (m-1)}{1 + 2 n},\qquad \ell = -(2n-2)  -\frac{2 m + 2 n -1}{2 (m-1)}$,

\item $\displaystyle k = -(2m-2) -\frac{2 m + 2 n -1}{2 (n -1)},\qquad \ell  = -(2n-2) + \frac{2 (n-1)}{1 + 2 m}$.

\end{enumerate}

\end{thm}

The proof is similar to the proof of Theorem \ref{coin:d} and is omitted.

\section{Some rational cosets}

By composing the map $V^k(\gs\go_{2n}) \rightarrow V^k(\gs\go_{2n+1})$ with the quotient map $V^k(\gs\go_{2n+1}) \rightarrow L_k(\gs\go_{2n+1})$, we obtain an embedding
$$\tilde{V}^k(\gs\go_{2n}) \hookrightarrow L_k(\gs\go_{2n+1}),$$ where $\tilde{V}^k(\gs\go_{2n})$ denotes the quotient of $V^k(\gs\go_{2n})$ by the kernel $\cJ_k$ of the above composition. In general, it is a difficult and important problem to determine when $\cJ_k$ is the maximal proper graded ideal, or equivalently, when $\tilde{V}^k(\gs\go_{2n}) = L_k(\gs\go_{2n})$. In the case where $k$ is an admissible level for $\widehat{\gs\go_{2n}}$, Lemma 2.1 of \cite{ACKL} would then imply that
$\text{Com}(L_k(\gs\go_{2n}), L_k(\gs\go_{2n+1}))$ is simple, and hence its orbifold $\text{Com}(L_k(\gs\go_{2n}), L_k(\gs\go_{2n+1}))^{\mathbb{Z}_2}$ would be simple as well \cite{DLM}. Additionally, Theorem 8.1 of \cite{CLII} would imply that $\text{Com}(L_k(\gs\go_{2n}), L_k(\gs\go_{2n+1}))^{\mathbb{Z}_2}$ coincides with the simple quotient $\cD_k(n)$ of $\cD^k(n)$. This is particularly interesting in the cases where $\cD_k(n)$ is strongly rational.

We conclude by proving this for first family in Theorem \ref{coin:c}. These are new examples of cosets of non-rational vertex algebras by admissible level affine vertex algebras, which are strongly rational.

\begin{lemma} \label{main:application} For $n\geq 2$ and $m\geq 0$, we have an embedding of simple affine vertex algebras $$L_k(\gs\go_{2n}) \hookrightarrow L_k(\gs\go_{2n+1}),\qquad k = -(2n-2) + \frac{1}{2} (2 n + 2 m -1).$$
\end{lemma}

\begin{proof} We proceed by induction on $m$. In the case $m=0$, we have $k  = -n + \frac{3}{2}$, and it is well known that there exists a conformal embedding $L_k(\gs\go_{2n}) \hookrightarrow L_{k}(\gs\go_{2n+1})$, see e.g. Section 3 of \cite{AKMPP}. Next, we assume the result for $m-1$, so that
$k = -n + \frac{3}{2} + m-1$. Recall that the rank $2n+1$ free fermion algebra $\cF(2n+1)$ admits an action of $L_1(\gs\go_{2n+1})$, as well as an action of $L_1(\gs\go_{2n})$ via the embedding $L_1(\gs\go_{2n}) \hookrightarrow L_1(\gs\go_{2n+1})$. The image of $L_1(\gs\go_{2n})$ lies in the subalgebra $\cF(2n) \subseteq \cF(2n+1)$.

Since $k$ is admissible for $\gs\go_{2n+1}$, it is known \cite{KW} that we have a diagonal embedding of simple affine vertex algebras 
\begin{equation} \label{affembedd} L_{k+1}(\gs\go_{2n+1}) \hookrightarrow L_{k}(\gs\go_{2n+1}) \otimes \cF(2n+1).\end{equation} 

By induction, we have the map $L_k(\gs\go_{2n}) \hookrightarrow L_{k}(\gs\go_{2n+1})$. Then we have an embedding
\begin{equation} \label{affembedd2} L_{k+1}(\gs\go_{2n}) \hookrightarrow L_k(\gs\go_{2n}) \otimes \cF(n) \hookrightarrow L_{k}(\gs\go_{2n+1}) \otimes \cF(2n+1),\end{equation} where $\cF(2n) \hookrightarrow \cF(2n+1)$ is the isomorphism onto the first $2n$ copies.
Since the image of \eqref{affembedd2} lies in the image of \eqref{affembedd}, it follows that $L_{k+1}(\gs\go_{2n})$ embeds in $L_{k+1}(\gs\go_{2n+1})$. \end{proof} 

This has the following immediate consequence.

\begin{cor} For $n\geq 2$, $m\geq 0$, and $k = -(2n-2) + \frac{1}{2} (2 n + 2 m -1)$, we have an isomorphism
$$\text{Com}(L_k(\gs\go_{2n}) , L_k(\gs\go_{2n+1}))^{\mathbb{Z}_2} \cong \cW_{\ell}(\gs\go_{2m+1}),\qquad \displaystyle  \ell = -(2m-1) + \frac{2 m + 2 n -1}{2 m + 2 n+1}.$$ In particular, $\text{Com}(L_k(\gs\go_{2n}) , L_k(\gs\go_{2n+1}))^{\mathbb{Z}_2}$ is strongly rational.
\end{cor}

\begin{proof} This follows from Theorem \ref{coin:c} together with the fact that $\text{Com}(L_k(\gs\go_{2n}) , L_k(\gs\go_{2n+1}))^{\mathbb{Z}_2}$ is simple, and the map $\cD^k(n) \ra \text{Com}(L_k(\gs\go_{2n}) , L_k(\gs\go_{2n+1}))^{\mathbb{Z}_2}$ is surjective.
\end{proof}

Recall that the category of ordinary modules of an affine vertex algebra at admissible level is semisimple \cite{ArIII} and a vertex tensor category \cite{CHY}. Conjecturally, this category is fusion \cite{CHY} and this has been proven for simply-laced Lie algebras \cite{C}. For type $\gs\go_{2n+1}$ and level $k = -(2n-2) + \frac{1}{2} (2 n + 2 m -1)$ this conjecture is also true. First, $\text{Com}(L_k(\gs\go_{2n}) , L_k(\gs\go_{2n+1}))$ is a simple current extension, call it $\cV_\ell(\gs\go_{2m+1})$, of $\cW_{\ell}(\gs\go_{2m+1})$ and thus rational as well \cite{Li}. It follows that $L_k(\gs\go_{2n+1})$ is a simple $\mathbb Z$-graded extension of  $L_k(\gs\go_{2n}) \otimes \cV_\ell(\gs\go_{2m+1})$ in a rigid vertex tensor category $\mathcal C$ of $L_k(\gs\go_{2n}) \otimes \cV_\ell(\gs\go_{2m+1})$-modules, namely the Deligne product of the categories of ordinary $L_k(\gs\go_{2n})$-modules and $\cV_\ell(\gs\go_{2m+1})$-modules. Every ordinary module for $L_k(\gs\go_{2n+1})$ must be an object in this category $\mathcal C$. This means that as a braided tensor category the category of ordinary modules of  $L_k(\gs\go_{2n+1})$ is equivalent to the category of local modules for $L_k(\gs\go_{2n+1})$ viewed as an algebra object in $\mathcal C$ \cite{HKL, CKM1}.
All assumptions of Theorem 5.12 of \cite{CKM2} are satisfied (with $U=\cV_\ell(\gs\go_{2m+1})$ and $V=L_k(\gs\go_{2n})$) and so 
\begin{cor}
The category of ordinary modules of $L_k(\gs\go_{2n+1})$ at level $k = -(2n-2) + \frac{1}{2} (2 n + 2 m -1)$ is fusion. 
\end{cor}


\begin{thebibliography}{ABKS}


\bibitem[ArI]{ArI} T. Arakawa, \textit{Associated varieties of modules over Kac-Moody algebras and $C_2$-cofiniteness of $\cW$-algebras}, Int. Math. Res. Not. (2015)  Vol.\ 2015 11605-11666.
\bibitem[ArII]{ArII} T. Arakawa, \textit{Rationality of $\cW$-algebras: principal nilpotent cases}, Ann. Math. 182 (2015), no. 2, 565-694.


\bibitem[ArIII]{ArIII} T. Arakawa, \textit{Rationality of admissible affine vertex algebras in the category ${\mathcal{O}}$}, Duke Math. J. \textbf{165} (2016) no.1, 67-93

\bibitem[Ad]{Ad} D Adamovic, \textit{Rationality of Neveu-Schwarz vertex operator superalgebras}, Int. Math. Res. Not., 1997:865-874, 1997.

\bibitem[AKMPP]{AKMPP}
D.~Adamovi\'c, V.~G.~Kac, P.~M\"oseneder Frajria, P.~Papi and O.~Per\v{s}e, \textit{Finite vs. infinite decompositions in conformal embeddings},
Comm. Math. Phys. \textbf{348} (2016) no. 2, 445-473.

\bibitem[ACL]{ACL} T. Arakawa, T. Creutzig, and A. Linshaw, \textit{W-algebras as coset vertex algebras}, Invent. Math. 218 (2019), no. 1, 145-195.

\bibitem[ACKL]{ACKL} T. Arakawa, T. Creutzig, K. Kawasetsu, and A. Linshaw, \textit{Orbifolds and cosets of minimal $\cW$-algebras}, Comm. Math. Phys. 355, No. 1 (2017), 339-372.
 
 \bibitem[ALY]{ALY} T. Arakawa, C. H. Lam, and H. Yamada, \textit{Parafermion vertex operator algebras and $\cW$-algebras}, Trans. Amer. Math. Soc. 371 (2019), no. 6, 4277-4301.
 
 
\bibitem[C]{C} T. Creutzig, \textit{Fusion categories for affine vertex algebras at admissible levels}, Selecta Math. (N.S.) 25 (2019), no. 2, Paper No. 27, 21 pp.

 
 \bibitem[CGN]{CGN}
T.~Creutzig, N.~Genra and S.~Nakatsuka, \textit{Duality of subregular W-algebras and principal W-superalgebras}, Adv. Math. 383 (2021), 107685, 52 pp.
 
 
\bibitem[CHR]{CHR}
T.~Creutzig, Y.~Hikida and P.~B.~R\o{}nne,
\textit{$N=1$ supersymmetric higher spin holography on AdS$_3$}, JHEP \textbf{02} (2013), 019. 
 
\bibitem[CHY]{CHY}
T.~Creutzig, Y.~Z.~Huang and J.~Yang, \textit{Braided tensor categories of admissible modules for affine Lie algebras},
Comm. Math. Phys. \textbf{362} (2018) no.3, 827-854. 
 
\bibitem[CKMI]{CKM1}
T.~Creutzig, S.~Kanade and R.~McRae, \textit{Tensor categories for vertex operator superalgebra extensions}, arXiv:1705.05017, to appear in Memoirs AMS.
 
\bibitem[CKMII]{CKM2}
T.~Creutzig, S.~Kanade and R.~McRae, \textit{Glueing vertex algebras}, arXiv:1906.00119. 
 
 
 
\bibitem[CLI]{CLI} T. Creutzig and A. Linshaw, \textit{The super $\cW_{1+\infty}$ algebra with integral central charge}, Trans. Amer. Math. Soc. 367 (2015), no. 8, 5521-5551.
\bibitem[CLII]{CLII} T. Creutzig and A. Linshaw, \textit{Cosets of affine vertex algebras inside larger structures}, J. Algebra 517 (2019) 396-438.

\bibitem[CLIII]{CLIII} T. Creutzig and A. Linshaw, \textit{Trialities of $\mathcal{W}$-algebras}, arXiv:2005.10234.

\bibitem[CV]{CV}
C.~Candu and C.~Vollenweider,
The $\mathcal{N} =$ 1 algebra $\mathcal{W}_\infty[\mu]$ and its truncations,
JHEP \textbf{11} (2013), 032.

\bibitem[DSK]{DSK} A. De Sole and V. Kac, \textit{Freely generated vertex algebras and non-linear Lie conformal algebras}, Comm. Math. Phys. 254 (2005), no. 3, 659-694. 

\bibitem[DLY]{DLY} C. Dong, C. Lam, and H. Yamada, \textit{$\cW$-algebras related to parafermion algebras}, J. Algebra 322 (2009), no. 7, 2366-2403. 

\bibitem[DLM]{DLM} C. Dong, H. Li, and G. Mason, \textit{Compact automorphism groups of vertex operator algebras}, Internat. Math. Res. Notices 1996, no. 18, 913--921.


 \bibitem[GL]{GL} N. Genra and A. Linshaw, \textit{Ito's conjecture and the coset realization of $\cW^k(sl(3|2))$}, arXiv:1901.02397, to appear in RIMS Kokyuroku Bessatsu.

\bibitem[HKL]{HKL}
Y.~Z.~Huang, A.~Kirillov and J.~Lepowsky, \textit{Braided tensor categories and extensions of vertex operator algebras},
Comm. Math. Phys. \textbf{337} (2015) no.3, 1143-1159.

\bibitem[I]{I} K. Ito, \textit{Quantum Hamiltonian reduction and $N=2$ coset models}, Phys. Lett. B 259 (1991) 73-78.





\bibitem[KW]{KW} V. G. Kac and M. Wakimoto, \text{Branching functions for winding subalgebras and tensor products}, Acta Appl. Math. 21(1-2): 3-39, 1990.



\bibitem[KL]{KL} S. Kanade and A. Linshaw, \textit{Universal two-parameter even spin $\cW_{\infty}$-algebra}, Adv. Math. 355 (2019), 106774, 58pp. 


\bibitem[Li]{Li} H. Li, \textit{Extension of vertex operator algebras by a self-dual simple module}, J. Algebra 187 (1997) 236-267.


\bibitem[LI]{LI} A. Linshaw, \textit{Invariant theory and the Heisenberg vertex algebra}, Int. Math. Res. Notices, 17 (2012), 4014-4050.
\bibitem[LII]{LII} A. Linshaw, \textit{Invariant subalgebras of affine vertex algebras}, Adv. Math. 234 (2013), 61-84.


\bibitem[LIII]{LIII} A. Linshaw, \textit{Universal two-parameter $\cW_{\infty}$-algebra and vertex algebras of type $\cW(2,3,\dots, N)$}, Compos. Math. 157 (2021), no. 1, 12-82.



\bibitem[W]{W} H. Weyl, \textit{The Classical Groups: Their Invariants and Representations}, Princeton University Press, 1946.

\end{thebibliography}
\end{document}